\documentclass[a4paper, reqno]{amsart}
\usepackage[T1]{fontenc}
\usepackage{lmodern}
\usepackage{amsmath,amsthm,amssymb,mathtools}
\usepackage[nobysame, alphabetic]{amsrefs}
\usepackage[margin=32mm,marginpar=3cm]{geometry}

\usepackage[final]{microtype}
\usepackage{enumerate}
\usepackage{tikz}
\usetikzlibrary{cd, arrows}
\usepackage{tikz-cd}
\usepackage{url}
\usepackage{xcolor}
\usepackage{relsize}
\usepackage{verbatim}

\definecolor{darkblue}{rgb}{0,0,0.6}

\usepackage[breaklinks, pdftex, ocgcolorlinks,colorlinks=true, citecolor=darkblue, filecolor=darkblue, linkcolor=darkblue, urlcolor=black]{hyperref}

\usepackage[capitalize,noabbrev]{cleveref}

\makeatletter
\newtheorem*{rep@theorem}{\rep@title}
\newcommand{\newreptheorem}[2]{%
	\newenvironment{rep#1}[1]{%
		\def\rep@title{#2 \ref{##1}}%
		\begin{rep@theorem}}%
		{\end{rep@theorem}}}
\makeatother

\numberwithin{equation}{section}

\newtheorem{corollary}[equation]{Corollary}
\newtheorem{lemma}[equation]{Lemma}
\newtheorem{conjecture}{Conjecture}

\newtheorem*{theorem*}{Theorem}

\newreptheorem{theorem}{Theorem}
\newreptheorem{conjecture}{Conjecture}
\newreptheorem{lemma}{Lemma}
\newreptheorem{proposition}{Proposition}
\newreptheorem{corollary}{Corollary}

\newtheorem{thmx}{Theorem}

\newtheorem{corx}[thmx]{Corollary}

\crefname{thmx}{Theorem}{Theorems}

\theoremstyle{definition}
\newtheorem{definition}[equation]{Definition}
\newtheorem{question}[equation]{Question}
\newtheorem{example}[equation]{Example}

\theoremstyle{remark}
\newtheorem{remark}[equation]{Remark}

\newtheorem*{remark*}{Remark}

\AddToHook{env/theorem/begin}{\crefalias{equation}{theorem}}
\AddToHook{env/thmx/begin}{\crefalias{equation}{thmx}}
\AddToHook{env/corx/begin}{\crefalias{equation}{corx}}
\AddToHook{env/assumption/begin}{\crefalias{equation}{assumption}}
\AddToHook{env/hypothesis/begin}{\crefalias{equation}{hypothesis}}
\AddToHook{env/corollary/begin}{\crefalias{equation}{corollary}}
\AddToHook{env/definition/begin}{\crefalias{equation}{definition}}
\AddToHook{env/lemma/begin}{\crefalias{equation}{lemma}}
\AddToHook{env/question/begin}{\crefalias{equation}{question}}
\AddToHook{env/example/begin}{\crefalias{equation}{example}}
\AddToHook{env/conjecture/begin}{\crefalias{equation}{conjecture}}
\AddToHook{env/remark/begin}{\crefalias{equation}{remark}}
\AddToHook{env/const/begin}{\crefalias{equation}{const}}

\crefname{assumption}{Assumption}{Assumptions}
\crefname{hypothesis}{Hypothesis}{Hypotheses}
\crefname{theorem}{Theorem}{Theorems}
\crefname{thmx}{Theorem}{Theorems}
\crefname{proposition}{Proposition}{Propositions}
\crefname{corollary}{Corollary}{Corollaries}
\crefname{definition}{Definition}{Definitions}
\crefname{lemma}{Lemma}{Lemmas}
\crefname{question}{Question}{Questions}
\crefname{example}{Example}{Examples}
\crefname{conjecture}{Conjecture}{Conjectures}
\crefname{remark}{Remark}{Remarks}
\crefname{const}{Construction}{Constructions}

\numberwithin{equation}{section}

\newcommand{\bigast}{\mathop{\Huge \mathlarger{\mathlarger{\ast}}}}

\newcommand{\Q}{\mathbb{Q}}

\newcommand{\Z}{\mathbb{Z}}

\newcommand{\Id}{\operatorname{Id}}

\newcommand{\ol}{\overline}

\newcommand{\sm}{\setminus}

\newcommand{\ks}{\operatorname{ks}}

\DeclareMathOperator{\Wh}{Wh}

\title{Some remarks on $h$-cobordisms between smooth 4-manifolds}

\author{Alexander Kupers}
\address{Department of Computer and Mathematical Sciences, University of Toronto Scarborough, Ontario, Canada}
\email{a.kupers@utoronto.ca}

\author{Mark Powell}
\address{School of Mathematics and Statistics, University of Glasgow, United Kingdom}
\email{mark.powell@glasgow.ac.uk}

\begin{document}
	
\begin{abstract}
It is not known whether the realisation part of the $s$-cobordism theorem holds for smooth $4$-manifolds, nor whether every pair of smoothly $h$-cobordant 4-manifolds is also smoothly $s$-cobordant. We provide some new conditions under which these questions admit a positive answer. We also give conditions under which the `standard' method to construct an $h$-cobordism with specified torsion cannot work.
\end{abstract}

	\def\subjclassname{\textup{2020} Mathematics Subject Classification}
	\expandafter\let\csname subjclassname@1991\endcsname=\subjclassname

\maketitle
	
\vspace{-.5cm}
	
\section{Introduction}

A smooth cobordism $W \colon M \leadsto N$ between compact smooth $d$-manifolds (possibly with boundary) is an \emph{$h$-cobordism} if $\partial W = M \cup (\partial M \times I) \cup N$ and the inclusions $M \hookrightarrow W$ and $N \hookrightarrow W$ are homotopy equivalences; in this case we say that $M$ and $N$ are \emph{smoothly $h$-cobordant}. If moreover both inclusions are simple homotopy equivalences, then $W \colon M \leadsto N$  is an \emph{$s$-cobordism}, and we say $M$ and $N$ are \emph{smoothly $s$-cobordant}.

To each $h$-cobordism, throughout assumed to be connected and based, we can associate its \emph{torsion} $\tau(W,M) \in \Wh(\pi_1(M))$, taking values in the Whitehead group of the fundamental group~\cite[\S 2]{Whitehead-she}. By Milnor's duality formula~\cite[p.~394]{Mi66} for torsion, $\tau(W,N) =0$ if and only if $\tau(W,M)=0$, which by a result of Whitehead~\cite[\S 10]{Whitehead-she} in turn holds if and only if $W$ is an $s$-cobordism. In dimensions $d \geq 5$, the Whitehead torsion induces a bijection \cite{Kervaire}
\[\frac{\{\text{smooth $h$-cobordisms $W \colon M \leadsto N$}\}}{\text{diffeomorphisms fixing $M \cup (\partial M \times I)$ pointwise}} \overset{\cong}{\longrightarrow} \Wh(\pi_1(M)).\]
In dimension $d=4$, it is well-known that this map is not injective \cite{Donaldson}, and it is an open question whether it is surjective~\cite{K3}*{Problem 4.18}. Let us make the latter explicit.

\begin{question}\label{question-realisation}
	Given a connected, compact, smooth 4-manifold $M$ with fundamental group $\pi = \pi_1(M)$, can each torsion $\tau \in \Wh(\pi)$ be realised by a smooth $h$-cobordism?
\end{question}

The following related question was raised explicitly in \cite{KPR-table}*{Question~1.16} and \cite{K3}*{Problem 4.18}, and also remains open.

\begin{question}\label{question-h-implies-s}
Is every pair of  compact, smooth 4-manifolds that are smoothly $h$-cobordant also smoothly $s$-cobordant?
\end{question}

If the Whitehead group $\Wh(\pi)$ vanishes, for example if $\pi$ is torsion-free and satisfies the $K$-theoretic Farrell-Jones conjecture \cite{Lueck-FJ-survey}, then every $h$-cobordism is an $s$-cobordism; restricting to 4-manifolds with this fundamental group, both \cref{question-realisation} and \ref{question-h-implies-s} trivially have positive answers. Finite cyclic groups are the best-known examples of groups for which the Whitehead group can be nontrivial \cite{Mi66,Ol88}. Our first result is that both questions admit a positive answer when certain technical conditions are satisfied, e.g.~for $\pi$ finite cyclic.

Recall that a smooth $h$-cobordism $W \colon M \leadsto M'$ is \emph{inertial} if $M$ and $M'$ are diffeomorphic, $\mathrm{SK}_1(\pi) \coloneq \ker [K_1(\Z\pi) \to K_1(\Q\pi)]$, and $L_5^{s}(\Z\pi)$  and $L_5^h(\Z\pi)$ are the Wall surgery obstruction groups, with $s$- and $h$-decoration respectively.

\begin{thmx}\label{thm:main}
  Let $M$ be a smooth, compact, connected, oriented 4-manifold with finite fundamental group $\pi \coloneq \pi_1(M)$. Suppose that either
  \begin{enumerate}[(i)]
    \item\label{case-i} $\mathrm{SK}_1(\pi) = 0$ and the forgetful map $L_5^{s}(\Z\pi) \to L_5^h(\Z\pi)$ is injective;
    \item\label{case-ii} $\mathrm{SK}_1(\pi) = 0$ and $M$ is \emph{topologically pre-stabilised}, that is, homeomorphic to  $X \# (S^2 \times S^2)$ for some closed topological 4-manifold $X$; or
    \item\label{case-iii} $M$ is \emph{smoothly pre-stabilised}, that is, diffeomorphic to  $X \# (S^2 \times S^2)$ for some closed smooth 4-manifold $X$.
  \end{enumerate}
 Then,  for each $x \in \Wh(\pi)$ there exists an inertial smooth $h$-cobordism $W \colon M \leadsto M'$ with torsion $\tau(W,M) = x$.
\end{thmx}

A key ingredient in the proof is Galvin's realisation of the Casson--Sullivan invariant~\cite{Galvin-CS}.
Next, a positive answer to \cref{question-realisation} that produces inertial $h$-cobordisms, as in \cref{thm:main}, leads to a positive answer to \cref{question-h-implies-s}.

\begin{corx}\label{cor:h-implies-s}
If $M$ satisfies the hypotheses of \cref{thm:main}, then every 4-manifold that is smoothly $h$-cobordant to $M$ is also smoothly $s$-cobordant to $M$.
\end{corx}

\begin{example}\,
	\begin{itemize}
		\item Case \eqref{case-i} applies if $\pi$ is a finite cyclic group $C_n$ of order $n$, since $\mathrm{SK}_1(\pi) = 0$ by \cite{Bass-Milnor-Serre}*{Proposition~4.14} and $L_5^s(\Z C_n) =0$ by \cite{Bak-computation-L-groups}.
		\item Case \eqref{case-ii} applies to many more examples of fundamental groups $\pi$, since $\mathrm{SK}_1(\pi)=0$ in for example the following cases:
		\begin{itemize}
			\item $\pi$ is abelian and either
			(a) each Sylow subgroup of $\pi$ has the form $C_{p^n}$ or  $C_p \times C_{p^n}$ for some $n$, or (b) $\pi = (C_2)^k$ for some $k$  \cite{Ol88}*{Theorem~14.2~(iii)},
			\item $\pi$ is a dihedral group $D_{2n}$ of order $2n$ \cite[Theorem 21.1]{MagurnThesis},
			\item $\pi$ is on the list of \cite[Theorem A]{Ushitaki}, e.g.~the binary icosahedral group.
		\end{itemize}
	\end{itemize}
\end{example}

\begin{remark}
	By the Ranicki--Rothenberg exact sequence of \cite{Shaneson-GxZ}*{Section~4}, the kernel of $L_5^{s}(\Z\pi) \to L_5^h(\Z\pi)$ is a quotient of the Tate cohomology group $\smash{\widehat{H}^6}(C_2;\Wh(\pi))$. Since the latter are 2-torsion, this kernel is trivial
 for example when $L_5^{s}(\Z\pi)$ is free abelian.
\end{remark}

We continue to consider \cref{question-realisation}, but now without any consideration of whether the realising $h$-cobordism is inertial. There is a `standard' method to try to construct a smooth $h$-cobordism $W \colon M \leadsto N$ with a given torsion $x$. We explain it in detail later, but roughly one represents $x$ by an $(n \times n)$-matrix $X$ with columns $X_i$, adds $n$ trivial 5-dimensional $2$-handles to $M \times I$, and then represents each $X_i$ by an embedded $2$-sphere by tubing together parallel copies of capped-off cores of the new $2$-handles. Attaching a 3-handle to each such embedded 2-sphere results in a cobordism $W \colon M \leadsto N$, where $M \hookrightarrow W$ is homotopy equivalence with $\tau(W,M) = [X] \in \Wh(\pi)$. But $N \hookrightarrow W$ is in general only a $\Z\pi$-homology isomorphism, as there is no reason a priori for the embedded 2-spheres representing the $X_i$ to admit immersed dual spheres and hence one loses control on the fundamental group $\pi_1(N)$.

Our second result says that whether this construction can work or not is closely related to an invariant of M.~Cohen~\cite{MCohen}. To each $x \in \Wh(\pi)$, he associates a number $\dim(x) \in \{0,2,3\}$, with $\dim(0)=0$ and $\dim(x) \in \{2,3\}$ for $x \neq 0$, and he conjectured that $x \neq 0$ $\Leftrightarrow$ $\dim(x) = 3$.

\begin{thmx}\label{theorem:obvious construction does not work}
  Let $M$ be a smooth, compact, connected, oriented 4-manifold with fundamental group $\pi \coloneq \pi_1(M)$. Then for $x \in \Wh(\pi)$ we have:
  \begin{enumerate}[(i)]
  	\item if $\dim(x) \in  \{0,2\}$, there exists an $h$-cobordism $W \colon M \leadsto N$ with $\tau(W,M) = x$;
  	\item if $\dim(x) = 3$, the standard construction \emph{cannot} yield any $h$-cobordism $W \colon M \leadsto N$ with $\tau(W,M) = x$.
  \end{enumerate}
\end{thmx}

\begin{example}\label{exam:introduction}
Rothaus and Magurn found examples of $x \in \Wh(D_{2n})$ with $\dim(x) =3$ \cite{RothausBull,Rothaus,Magurn}; see \cref{exam:rothaus} for more details. By \cref{theorem:obvious construction does not work} one cannot realise these torsions using the standard construction. \cref{thm:main} \eqref{case-ii}, however, does realise these torsions for topologically pre-stabilised manifolds, as does \eqref{case-iii} for smoothly pre-stabilised 4-manifolds. Since our proof of \cref{thm:main} relies on surgery theory and smoothing theory, it is non-constructive. We hope 4-manifold topologists are encouraged to find explicit constructions.
\end{example}

\subsubsection*{Acknowledgements}
The authors would like to thank Ian Hambleton and Daniel Kasprowski for several helpful remarks, including the suggestion to add case \eqref{case-iii} of \cref{thm:main}. The second author thanks Michelle Daher, Daniel Galvin, Csaba Nagy, John Nicholson, and Arunima Ray for stimulating conversations on these and related questions.  The authors are grateful to the organisers of the Zeeman Centenary conference at the University of Warwick in December 2025, at which some of this work was done. AK acknowledges the support of the Natural Sciences and Engineering Research Council of Canada (NSERC) [funding reference number 512156 and 512250].

\tableofcontents

\section{Proof of Theorem~\ref{thm:main} and Corollary~\ref{cor:h-implies-s}}\label{section:proof of theorem-main}

Recall that $M$ is a smooth, compact, connected, oriented $4$-manifold with finite fundamental group $\pi \coloneq \pi_1(M)$. Since $\pi$ is finite, it is a good group in the sense of \cite{FQ}, and hence by the topological realisation of Whitehead torsion by $h$-cobordisms (e.g.~\cite{KPR-table}*{Theorem~3.5}), for every $x \in \Wh(\pi)$ there exists a topological $h$-cobordism $W \colon M \leadsto N$ with $\tau(W,M) =x$. We want to show that $M$ and $N$ are homeomorphic. To do this, we deal with case~\eqref{case-i} of \cref{thm:main} first, in \cref{lemma:case-i}, and then we show it for cases~\eqref{case-ii} and \eqref{case-iii} simultaneously, in \cref{lemma:case-ii}. After that we complete the proof of \cref{thm:main} in unison for all three cases.


\begin{lemma}\label{lemma:case-i}
 In case~\eqref{case-i}, $M$ and $N$ are homeomorphic.
\end{lemma}

\begin{proof}
Let $g \colon W \to M$ be a homotopy inverse to the inclusion $\mathrm{inc}_M \colon M \to W$. By using the collar on $M$ we may assume that $g$ is a weak deformation retract, which in particular means that $g \circ \mathrm{inc}_M = \Id_M$.
Let $f \coloneq g \circ \mathrm{inc}_N \colon N \to M$ be the homotopy equivalence $f \colon N \to M$ induced by $W$. This has Whitehead torsion $\tau(f) = x- \ol{x}$ (e.g~\cite{NNP-he-vs-she}*{Proposition~2.38}), where $x \mapsto \ol{x}$ is the involution on the Whitehead group.
The involution acts trivially on the quotient of $\Wh(\pi)$ by $\mathrm{SK}_1(\pi)$ \cite{Ol88}*{Corollary~7.5}, so using the hypothesis that $\mathrm{SK}_1(\pi) = 0$, the involution acts trivially on $\Wh(\pi)$, and thus $\tau(f) = x-\ol{x} = x-x =0$, i.e.\ $f$ is a simple homotopy equivalence.

Take the product of $g$ with a continuous map $h \colon W \to I$ satisfying $h^{-1}(\{0\}) =M$ and $h^{-1}(\{1\}) = N$, to obtain a homotopy equivalence of pairs $(F,\Id_M \sqcup f) \colon (W,M \cup N) \to (M \times I,M \cup M)$. This can be considered as a degree one normal map over $M \times I$ by adding normal data as follows: take $G \colon M \times I \to W$ to be a homotopy inverse to $F$, and consider the stable vector bundle data
\[
\begin{tikzcd}
  \nu_W \ar[r] \ar[d] & G^*\nu_W \ar[d] \\[-5pt] W \ar[r,"F"] & M \times I,
\end{tikzcd}
\]
where $\nu_W$ is the stable normal bundle of $W$. (It may be helpful to observe this is similar to the construction of the map $\eta \colon \mathcal{S}^s(M \times I,\partial) \to \mathcal{N}(M \times I,\partial)$ in the surgery exact sequence.)

Since the restrictions of this degree one map to the ends are given by $f\colon N \to M$ and $\Id_M \colon M \to M$, both of which are simple homotopy equivalences, the simple surgery obstruction~$\sigma_s(W,F)$ of $(W,F)$ represents an element in the surgery obstruction group~$L_5^s(\Z\pi)$. Since $F$ is a homotopy equivalence, $(W,F)$ has trivial surgery obstruction in~$L_5^h(\Z\pi)$, and hence~$\sigma_s(W,F)$ lies in the kernel of $L_5^s(\Z\pi) \to L_5^h(\Z\pi)$. As this kernel is trivial by hypothesis~\eqref{case-i}, $\sigma_s(W,F)=0$, and so~$(W,F)$ is normally bordant relative to the boundary to a simple homotopy equivalence~\cite{Wall-SCM}*{Chapter~6}; the resulting bordism is an $s$-cobordism.  Since $\pi$ is finite, and finite groups are good, by the $s$-cobordism theorem~\cite{FQ}*{Chapter~7} it follows that $M$ and $N$ are homeomorphic.
\end{proof}


\begin{lemma}\label{lemma:case-ii}
In cases \eqref{case-ii} and \eqref{case-iii}, $M$ and $N$ are homeomorphic.
\end{lemma}

\begin{proof}
As $M$ and $N$ are $h$-cobordant, they are stably homeomorphic~\cite{Lawson-decomposing}. Since \eqref{case-ii} and \eqref{case-iii} require there is a homeomorphism $M \cong X \# (S^2 \times S^2)$, the fundamental group is finite, and $M$ as well as $N$ are closed, \cite{Hambleton-Kreck-93-II}*{Theorem~B} implies that $M$ and $N$ are homeomorphic.
\end{proof}

We have now shown that $M$ and $N$ are homeomorphic in all three cases. We record the following conclusion.

\begin{corollary}\label{cor:existence-of-top-inertial}
 Let $M$ be a smooth, compact, connected, oriented $4$-manifold with finite fundamental group $\pi \coloneq \pi_1(M)$, and suppose at least one of the assumptions  \eqref{case-i}, \eqref{case-ii}, or \eqref{case-iii} holds. For every $x \in \Wh(\pi)$, there exists a topological inertial $h$-cobordism $W \colon M \leadsto M$ with $\tau(W,M) =x$.
\end{corollary}


We will need the following in the proof of the next lemma.
Given a self-homeomorphism $\varphi \colon M \to M$ that is the identity on the boundary, let $\mathrm{C}_{\varphi}$ denote the mapping cylinder of $\varphi$,  with the evident smooth structure on the boundary.
The \emph{Casson--Sullivan invariant} $\mathrm{cs}({\varphi})$  of $\varphi$ is defined as the image of the Kirby-Siebenmann invariant $\ks(\mathrm{C}_{\varphi},\partial \mathrm{C}_{\varphi})$ under the isomorphisms
\[H^4(\mathrm{C}_{\varphi},\partial \mathrm{C}_{\varphi};\Z/2) \underset{\cong}{\xrightarrow{\mathrm{PD}}} H_1(\mathrm{C}_{\varphi};\Z/2) \underset{\cong}{\xleftarrow{\mathrm{inc}_*}} H_1(M;\Z/2),\]
where $\mathrm{inc} \colon M \to \mathrm{C}_{\varphi}$ is the inclusion into the mapping cylinder as the domain.
In our cases, Galvin showed that it is possible to realise each element of $H_1(M;\Z/2)$ as  $\mathrm{cs}({\varphi})$, for some $\varphi$.
For cases \eqref{case-i} and \eqref{case-ii}, apply \cite{Galvin-CS}*{Theorem~1.7}, using that the \emph{Casson--Sullivan realisability criterion} of \cite{Galvin-CS}*{Definition 4.11} is satisfied whenever $\mathrm{SK}_1(\pi)=0$, by \cite{Galvin-CS}*{Proposition~4.13}. For case \eqref{case-iii} we instead use \cite{Galvin-CS}*{Theorem~1.1}, which applies to smoothly pre-stabilised 4-manifolds.

\begin{lemma}\label{lemma:realising-CS}
Let $M$ and $N$ be compact, smooth 4-manifolds with fundamental group $\pi$, and let $(W;M,N)$ be a topological $h$-cobordism. Suppose that either $\pi$ is finite and $\mathrm{SK}_1(\pi)=0$, or $N$ is smoothly pre-stabilised. Then the given smooth structure on $M$, and some smooth structure on~$N$ diffeomorphic to the given one, extend to a smooth structure on $W$.
\end{lemma}

\begin{proof}
Start with the given smooth structure on $\partial W = M \cup_{\partial M} N$.
The Kirby--Siebenmann obstruction to smoothing $W$ relative to its boundary is
\[\ks(W,\partial W) \in H^4(W;\partial W;\Z/2). \]
We claim that we can reparametrise $N$ by precomposing the embedding $N \to W$ with the inverse of a suitable homeomorphism ${\varphi} \colon N \to N$, to obtain an $h$-cobordism $W \colon M \leadsto N$ with a new smooth structure on its boundary, so that now the obstruction to smoothing $W$ relative to its boundary vanishes. Assuming this, since the dimension of $W$ is at least $5$, high-dimensional smoothing theory~\cite[Essay IV]{Kirby-Siebenmann:1977-1} implies that this smooth structure on $M \cup_{\partial M} N$ extends to a smooth structure on $W$.
In addition, the new smooth structure on $N$ arises by pulling back the original smooth structure along a homeomorphism, and so this homeomorphism provide the claimed diffeomorphism between the new and old smooth structures on $N$.

We use the Casson--Sullivan realisation described before the lemma, together with a suitable additivity formula for the Kirby--Siebenmann invariant, to prove the claim. Reparametrising using the inverse of the homeomorphism ${\varphi} \colon N \to N$ is equivalent to replacing $W$ with the union $W' \coloneq W \cup_N \mathrm{C}_{\varphi}$.
Consider the following commutative diagram of maps of pairs:
\[
\begin{tikzcd}
(W,\partial W) \ar[d,swap,"\kappa"] & (W',\partial W' \cup N)  \ar[dl,swap,"q"] \ar[dr,"r"] & (\mathrm{C}_{\varphi},\partial \mathrm{C}_{\varphi}) \ar[d,"\lambda"] \\
(W',\partial W' \cup \mathrm{C}_{\varphi}) & (W',\partial W') \ar[u,"p"] \ar[l,"i"] \ar[r,swap,"j"] & (W',\partial W' \cup W).
\end{tikzcd}
\]
The maps $\kappa$ and $\lambda$ induce isomorphisms on cohomology by excision, and the maps $i$ and $j$ induce isomorphisms on cohomology because $W$ and $\mathrm{C}_{\varphi}$ are $h$-cobordisms.
By naturality of the Kirby--Siebenman  invariant we have that
\begin{align}\label{eqn:ks-equalities}
  \ks(W',\partial W') & = p^* \ks(W',\partial W' \cup N);\;\;\;\; \ks(W,\partial W) = \kappa^* \ks(W', \partial W' \cup \mathrm{C}_{\varphi}); \text{ and } \\ \ks(\mathrm{C}_{\varphi},\partial \mathrm{C}_{\varphi}) &= \lambda^* \ks(W',\partial W' \cup W). \nonumber
\end{align}
We therefore have the following in $H^4(W',\partial W';\Z/2)$. Here the first and last equations use \eqref{eqn:ks-equalities},
the second equation uses naturality of obstruction theory, and the third uses that $i = q \circ p$ and $j= r \circ p$.
 \begin{align*}
\ks(W',\partial W') &= p^* \ks(W',\partial W' \cup N) \\  &= p^* \big(q^* \ks(W', \partial W' \cup \mathrm{C}_{\varphi}) + r^* \ks(W',\partial W' \cup W)\big) \\ &= i^* \ks(W', \partial W' \cup \mathrm{C}_{\varphi}) + j^* \ks(W',\partial W' \cup W) \\  &=
i^* (\kappa^*)^{-1} \ks(W,\partial W) + j^* (\lambda^*)^{-1} \ks(\mathrm{C}_\varphi,\partial \mathrm{C}_\varphi).
\end{align*}
The fact that $j^* (\lambda^*)^{-1}$ is surjective, together with the aforementioned realisation result for Casson--Sullivan invariants, means that we can choose $\varphi$ so that $j^*(\lambda^*)^{-1}\ks(\mathrm{C}_\varphi,\partial \mathrm{C}_\varphi) = i^* (\kappa^*)^{-1} \ks(W,\partial W)$, and hence $\ks(W',\partial W') = 0$. To see that we can choose $\varphi$ appropriately, we compare the images of the invariants under Poincar\'{e} duality, in $H_1(W';\Z/2)$ (a group that does not depend on $\varphi$), and note that by Galvin's Casson--Sullivan realisation~\cite{Galvin-CS}*{Theorems~1.1~and~1.7},  we can obtain any element of $H_1(N;\Z/2) \cong H_1(\mathrm{C}_\varphi;\Z/2) \cong  H_1(W';\Z/2)$.  Here Galvin's Casson--Sullivan realisation applies because $\pi$ is finite and $\mathrm{SK}_1(\pi)=0$, or $N$ is smoothly pre-stabilised, as discussed above the statement of the lemma.
\end{proof}

Apply \cref{lemma:realising-CS} to an inertial $h$-cobordism $(W;M,M)$ arising from \cref{cor:existence-of-top-inertial}.
Let $M'$ denote $M$ with the second smooth structure arising from the lemma. By the lemma, the smooth structure on $M \cup M'$ extends to one on $W$.
Thus $W \colon M \leadsto M'$ is the desired inertial smooth $h$-cobordism with Whitehead torsion $x \in \Wh(\pi)$.

\smallskip

To deduce Corollary~\ref{cor:h-implies-s}, start with a smooth $h$-cobordism $V \colon M \leadsto N$, and let $y \in \Wh(\pi_1(M))$ denote its Whitehead torsion. Let $\theta \colon \pi_1(M) \to \pi_1(N)$ denote the isomorphism  induced by~$V$.  Set $x=-\theta_*(y) \in \Wh(\pi_1(N))$, and apply \cref{thm:main} to obtain a smooth inertial $h$-cobordism $W \colon N \leadsto N'$ with torsion $-\theta_*(y)$. Then the concatenation $W \circ V \colon M \leadsto N'$ has torsion $y +\theta_*^{-1}(-\theta_*(y)) = 0 \in \Wh(\pi_1(M))$, and so is a smooth $s$-cobordism.  Then, since $N'$ is diffeomorphic to $N$, by glueing on $N \times I$ using the diffeomorphism we deduce that in fact $M$ and $N$ are smoothly $s$-cobordant, as required.

\section{Proof of \cref{theorem:obvious construction does not work}}\label{section:obv-construction-fails}

We explain how an attempt to replicate, in dimension $d=4$,  the `standard' high-dimensional construction of $h$-cobordisms with specified torsion, leads to a group-theoretic conjecture of M.~Cohen. This discussion goes through in all categories of manifolds---smooth, PL, and topological---so does not shed light on questions specifically concerning smooth manifolds.

\begin{remark}\label{rem:fq-error}
	Freedman and Quinn~\cite{FQ}*{Section~7.1, p.~102} assert: ``The standard construction of $h$-cobordisms (Rourke and Sanderson~\cite{RS82}*{p.~90}) works in this dimension, and shows that there is an $h$-cobordism with any given torsion.'' Combining \cref{theorem:obvious construction does not work} and \cref{exam:introduction}, this statement is not correct. This minor error was observed before in \cite[Lemma 5.8]{HausmannJahren}, which also explained how to use \cite{FQ} to fix it (though they, in turn, erroneously omit the hypothesis that the fundamental group must be good and assert the proof works for smooth manifolds). See the survey \cite{KPR-table}*{Theorem~3.5} for a complete proof of topological realisation of Whitehead torsion by $h$-cobordisms.
\end{remark}

\subsection{A conjecture of M.~Cohen}
We start by recalling definitions from \cite[\S 1]{MCohen}. Let $\mathcal{P}$ denote the data of a group $\pi$, a free group $F = \langle x_1,\ldots,x_n \rangle$ on $n$ generators, and relators  $r_1,\ldots,r_n$ that lie in the smallest normal subgroup $N$ of $\pi \ast F$ containing $x_1,\ldots,x_n$. Sending $x_i$ to $e$ induces the right surjection, split by the inclusion of $\pi$, in the short exact sequence
\[1 \longrightarrow \frac{N}{(r_1,\ldots,r_n)} \longrightarrow \pi(\mathcal{P}) := \frac{\pi * F}{(r_1,\ldots,r_n)} \longrightarrow \pi \longrightarrow 1.\]
Such an extension is said to be \emph{trivial} if the right map $\pi(\mathcal{P}) \to \pi$ is an isomorphism, and  \emph{proper} if it is not. The relations $r_i$ can, like every element in $N$, be written as a product
\[r_i = \prod_{k=1}^{n(i)} g(i,k) x(i,k)^{\varepsilon(i,k)} g(i,k)^{-1},\]
with $n(i) \geq 0$, $g(i,k) \in \pi$, $x(i,k) \in \{x_1,\ldots,x_n\}$, and $\varepsilon(i,k) \in \pm 1$. Define then an $(n \times n)$-matrix $X(\mathcal{P})$ with entries in $\mathbb{Z}\pi$ by
\[X(\mathcal{P})_{ij} \coloneq \sum_{x(i,k) = x_j} \varepsilon(i,k) g(i,k).\]
See \cite{MCohen}*{\S 4} for more details as well as an alternative description of $X(\mathcal{P})$ in terms of Fox derivatives. The following lemma is a straightforward exercise.

\begin{lemma}\label{lem:from-matrix-to-P}
  Given a finitely presented group $\pi$ and an $(n \times n)$-matrix $X$ with entries in $\mathbb{Z}\pi$, there exists $\mathcal{P}$ such that $X(\mathcal{P}) = X$.
\end{lemma}

Following M.~Cohen, we say that $\mathcal{P}$ is \emph{admissible} if $X(\mathcal{P})$ is invertible. We are interested specifically in admissible $\mathcal{P}$, because in such cases $X(\mathcal{P})$ represents an element
\[\tau(\mathcal{P}) = [X(\mathcal{P})] \in \Wh(\pi).\]
By \cite[Proposition 4.1]{MCohen}, if the extension is trivial, i.e.\ $\pi(\mathcal{P}) \xrightarrow{\cong} \pi$ is an isomorphism, $X(\mathcal{P})$ is invertible and thus the extension is admissible. There exist, however, also admissible proper extensions, as we will see in \cref{exam:rothaus}.

\begin{definition}For $x \in \Wh(\pi)$, define
	\[\dim(x) = \begin{cases} 0 & \text{if $x = 0$,} \\
		2 & \text{if $x \neq 0$ and there exists a trivial admissible presentation with $\tau(\mathcal{P}) = x$,} \\
		3 & \text{if $x \neq 0$ and every admissible presentation with $\tau(\mathcal{P}) = x$ is proper}.\end{cases}\]
\end{definition}

\begin{example}\label{exam:rothaus} For $D_{2n}$ the dihedral group with $2n$ elements, there exist elements $x \in \Wh(D_{2n})$ with $\dim(x)=3$ whenever $n$ does not divide $12$. This was proven by Rothaus for $n$ prime \cite[Theorem 11]{Rothaus} and in general by Magurn \cite[Corollary 13 and 15]{Magurn}. An explicit example appears in \cite[p.~285]{RothausBull}: writing $D_{10} = \langle g,h \mid g^5,h^2,hgh = g^{-1} \rangle$, the unit $(-1+g-g^2+g^3+g^4) + h(1-2g+g^2) \in \Z D_{10}$ represents an element $x \in \Wh(D_{10})$ with $\dim(x) = 3$.\end{example}

\begin{remark}
The work of Rothaus and Magurn amounts to the statement that if $x \in \Wh(G)$ maps to a non-zero element in the quotient $\mathrm{Ro}(G) \coloneq \Wh(G)/\smash{\Wh^+(G)}$, where $\smash{\Wh^+(G)} $ is a certain subgroup of ``positive elements'', then $\dim(x) = 3$. Magurn observed the quotient map $\Wh(G) \to \mathrm{Ro}(G)$ factors through  $\Wh'(G)/2\Wh'(G)$, where $\Wh'(G)$ is the torsion-free quotient of $\Wh(G)$ \cite[\S 3]{Magurn} and that if $\mathrm{Ro}(H) = 0$ for all hyperelementary subgroups $H \subset G$ then $\mathrm{Ro}(G) = 0$ as well \cite[Theorem 2 (b)]{Magurn}. It follows that this approach cannot prove that $\dim(x) = 3$ for a majority of non-zero elements of $\Wh(G)$.
\end{remark}

Notwithstanding, there are no known examples of $x$ with $\dim(x) = 2$, and M.~Cohen conjectured \cite[Conjecture A]{MCohen} the following.

\begin{conjecture}[M.~Cohen]
For every group $\pi$, $\dim(x) = 3$ if and only if $x \neq 0 \in \Wh(\pi)$.
\end{conjecture}

\subsection{The standard construction}
To begin the proof of \cref{theorem:obvious construction does not work}, let us recall what is arguably the `standard' method for constructing an $h$-cobordism with specified torsion; it is alluded to in \cref{rem:fq-error} and used in \cite[Part 2, \S 12]{HudsonPL} and \cite[p.~82]{RS82}. Let $M$ be a connected compact smooth 4-manifold with $\pi = \pi_1(M)$ and $x \in \Wh(\pi)$ be represented by $X \in \mathrm{GL}_n(\Z \pi)$. Our candidate $W \colon M \leadsto N$ will be a concatenation  $W \coloneq W_1 \circ W_0 \colon M \leadsto N$ of two cobordisms. The first cobordism is
\[W_0 = (M \times I) \natural (S^2 \times D^3)^{\natural n} \colon M \leadsto M \# n(S^2 \times S^2)\]
and the second cobordism $W_1 \colon M \# (S^2 \times S^2)^{\# n}\ \leadsto N$ we shall construct now. We have $H_2(M \# n(S^2 \times S^2);\Z\pi) \cong H_2(M;\Z\pi) \oplus (\Z\pi \oplus \Z\pi)^n$ where the first copy of $\Z\pi$ in the $i$th $\Z\pi \oplus \Z\pi$ summand is represented by the embedded 2-sphere $A_i = S^2 \times \{0\}$ and the second by $B_i = \{0\} \times S^2$, intersecting transversely in a single point. For each column $X_j$ of $X$, we want to construct a $2$-sphere $C_j$ with trivialised normal bundle whose collection of $\Z\pi$-equivariant intersection numbers with the $2$-spheres $A_i$ is given by the column $X_j$, by tubing together parallel copies of the $B_i$. If we do so, attaching 5-dimensional $3$-handles along the $C_j$ yields a second cobordism
\[W_1 \colon M \# n(S^2 \times S^2) \leadsto N,\]
so that $M \hookrightarrow W = W_1 \circ W_0$ is a homotopy equivalence with torsion $\tau(W,M)$ given by $x$. The difficulty, which does not arise in higher dimensions, is to guarantee that the inclusion $N \to W$ is also a homotopy equivalence; note that by Poincar\'{e}--Lefschetz duality and the universal coefficients theorem, it is always a $\Z\pi$-homology equivalence, so this is a question about the induced map on fundamental groups.

To investigate this difficulty, we need prescribe the tubing, and to do so we use an admissible presentation $\mathcal{P}$ with generators $x_1,\ldots,x_n$ and relators $r_1,\ldots,r_n$, satisfying $[X(\mathcal{P})] = [X] \in \Wh(\pi)$, which exists by \cref{lem:from-matrix-to-P}. Recall that we can write
\[r_i = \prod_{k=1}^{n(i)} g(i,k) x(i,k)^{\varepsilon(i,k)} g(i,k)^{-1},\]
where $n(i) \geq 1$ since $X$ is invertible.
We may and will assume---by conjugating, taking the inverse, and reordering, if necessary---that $x(i,1) = x_i$, $g(i,1) = e$, and $\varepsilon(i,1) = 1$, for each~$i$.
Here we use the following observations.
\begin{enumerate}
 \item Replacing $r_i$ by $\gamma r_i^{\delta} \gamma^{-1}$, for $\delta \in \{\pm 1\}$ and $\gamma \in \pi$, corresponds to multiplying the $i$th row of $X(\mathcal{P})$ by $\delta\gamma$. This does not affect $[X(\mathcal{P})]$.
  \item Changing the order of the factors $g(i,k) x(i,k)^{\varepsilon(i,k)} g(i,k)^{-1}$ in $r_i$ does not affect $X(\mathcal{P})$, hence does not affect $[X(\mathcal{P})]$.
\end{enumerate}

We start by taking the 2-sphere $B(i,1)$ to be $B_i$, and for $k \geq 2$. we take $B(i,k)$ to be a disjoint parallel copy of $B_i$, each contained in the $i$th copy of $(S^2 \times S^2)^\circ$ in $M \# n(S^2 \times S^2)$ where $(-)^\circ$ denotes that we remove the interior of a $4$-disc. Identifying $(S^2 \times S^2) \sm \smash{\cup_{k=1}^{n(i)} B(i,k)}$ with $S^2 \times (S^2 \sm \{\text{$n(i)$ points}\})$ and observing that removing the interior of a $4$-disc does not change the fundamental group, we get an isomorphism
\[\pi_1\Big((S^2 \times S^2)^\circ \sm \cup_{k=1}^{n(i)} B(i,k)\Big) \cong \left\langle x(i,1),\ldots,x(i,n(i)) \,\middle \vert\, {\textstyle \prod}_{k=1}^{n(i)} x(i,k) \right\rangle.\]
This group is free, but it is helpful for the upcoming computation to present it in this manner, as having $n(i)$ generators and a single relation. Using Seifert--van Kampen, this implies
\[\pi_1\Big(M \# n(S^2 \times S^2) \sm \cup_{i=1}^n \cup_{k=1}^{n(i)} B(i,k)\Big) \cong \pi \ast \bigast_{i=1}^n \left\langle x(i,1),\ldots,x(i,n(i)) \,\middle \vert\, {\textstyle \prod}_{k=1}^{n(i)} x(i,k) \right\rangle\]
Here the $x(i,k)$ are represented by meridians of the removed 2-spheres, connected to the basepoint in $M$ by specified paths. Next, for $k \geq 2$ we tube $B(i,k)$ to $B(j_{i,k},1)$, where $j_{i,k}$ is defined by $x(i,k) = x_{j_{i,k}}$. For the tubing, use mutually disjoint, embedded paths, given as follows: go from $B(i,k)$ to the basepoint along the specified path, follow the loop $g(i,k)^{-1}$, and then go from the basepoint to $B(j_{i,k},1)$, again using the specified path. The tubing preserves the orientation if $\varepsilon(i,k)$ is $1$ and reverses it if $\varepsilon(i,k)$ is $-1$. The end result is $n$ embedded $2$-spheres $C_i$.

\begin{figure}
	\centerline{
		\def\svgwidth{5in}
\begingroup%
  \makeatletter%
  \providecommand\color[2][]{%
    \errmessage{(Inkscape) Color is used for the text in Inkscape, but the package 'color.sty' is not loaded}%
    \renewcommand\color[2][]{}%
  }%
  \providecommand\transparent[1]{%
    \errmessage{(Inkscape) Transparency is used (non-zero) for the text in Inkscape, but the package 'transparent.sty' is not loaded}%
    \renewcommand\transparent[1]{}%
  }%
  \providecommand\rotatebox[2]{#2}%
  \newcommand*\fsize{\dimexpr\f@size pt\relax}%
  \newcommand*\lineheight[1]{\fontsize{\fsize}{#1\fsize}\selectfont}%
  \ifx\svgwidth\undefined%
    \setlength{\unitlength}{922.00898899bp}%
    \ifx\svgscale\undefined%
      \relax%
    \else%
      \setlength{\unitlength}{\unitlength * \real{\svgscale}}%
    \fi%
  \else%
    \setlength{\unitlength}{\svgwidth}%
  \fi%
  \global\let\svgwidth\undefined%
  \global\let\svgscale\undefined%
  \makeatother%
  \begin{picture}(1,0.4158607)%
    \lineheight{1}%
    \setlength\tabcolsep{0pt}%
    \put(0,0){\includegraphics[width=\unitlength,page=1]{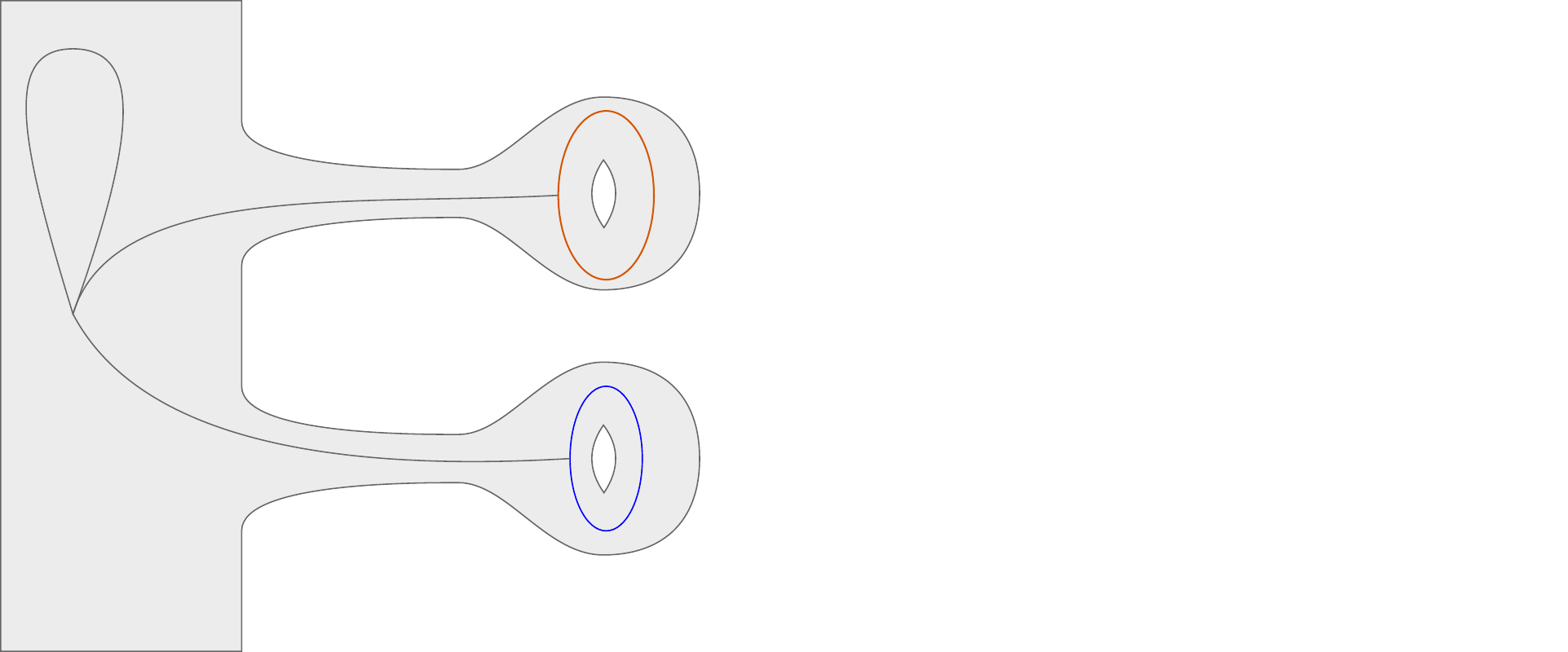}}%
    \put(0.35016802,0.36117607){\color[rgb]{0.4,0.4,0.4}\makebox(0,0)[lt]{\lineheight{1.25}\smash{\begin{tabular}[t]{l}$S^2 \times S^2$\end{tabular}}}}%
    \put(0.34993916,0.19167626){\color[rgb]{0.4,0.4,0.4}\makebox(0,0)[lt]{\lineheight{1.25}\smash{\begin{tabular}[t]{l}$S^2 \times S^2$\end{tabular}}}}%
    \put(0.47878057,0.20922971){\color[rgb]{0.4,0.4,0.4}\makebox(0,0)[lt]{\lineheight{1.25}\smash{\begin{tabular}[t]{l}$\leadsto$\end{tabular}}}}%
    \put(0.04608467,0.05779066){\color[rgb]{0.4,0.4,0.4}\makebox(0,0)[lt]{\lineheight{1.25}\smash{\begin{tabular}[t]{l}$M$\end{tabular}}}}%
    \put(0,0){\includegraphics[width=\unitlength,page=2]{construction.pdf}}%
  \end{picture}%
\endgroup%

	}
	\caption{A parallel copy of $B_1$ (in red) is tubed to a parallel copy of $B_2$ (in blue) along a path given concatenating a preferred path from the former to the basepoint, a loop in $M$, and a preferred path from the basepoint to the latter. In general we perform many such tubings, possibly involving many parallel copies of the same cores, and also need to specify with which orientation the tubings should be performed.}
\end{figure}

Using Seifert--van Kampen (see e.g.~\cite{Boyle}*{Lemma~9}), each tubing has the effect of adding a relation $x(i,k) = g(i,k) x(j_{i,k},1)^{\varepsilon(i,k)} g(i,k)^{-1}$.  Write
\[t(i,k) := g(i,k) x_{j_{i,k}}^{\varepsilon(i,k)} g(i,k)^{-1} = g(i,k) x(j_{i,k},1)^{\varepsilon(i,k)} g(i,k)^{-1} = g(i,k) x(i,k)^{\varepsilon(i,k)} g(i,k)^{-1}.\]
Note that $t(i,1) = x(i,1) = x_1$ and that
 \[\prod_{k=1}^{n(i)} t(i,k) = \prod_{k=1}^{n(i)} g(i,k) x(i,k)^{\varepsilon(i,k)} g(i,k)^{-1} = r_i.\]
Let $\langle y_1 \dots y_\ell \mid s_1,\dots,s_m \rangle$ be a presentation of $\pi$.
 Introducing this and the relations coming from tubing, we obtain
\begin{align*}
    &\pi_1\Big(M \# n(S^2 \times S^2) \sm \cup_{i=1}^n C_i \Big) \cong \\
  &\left\langle \big(y_a\big)_{a=1}^{\ell}, \big(x(i,k)\big)_{k=1}^{n(i)} \, \middle\vert \, \big(s_b\big)_{b=1}^m, {\textstyle \prod}_{k=1}^{n(i)} x(i,k), \big(x(i,k) = t(i,k)\big)_{k=2}^{n(i)} \right\rangle,
\end{align*}
where in each generator or relator in which $i$ appears, $i$ ranges from $1$ to $n$.
The relations $x(i,k) = t(i,k)$ can be cancelled against the generators $x(i,k)$, for $k=2,\dots,n(i)$ and $i=1,\dots,n$.   When doing this, we must substitute using these relations into  ${\textstyle \prod}_{k=1}^{n(i)} x(i,k)$. This yields
\[\prod_{k=1}^{n(i)} x(i,k) = x(i,1) \cdot \prod_{k=2}^{n(i)} x(i,k) = t(i,1) \cdot \prod_{k=2}^{n(i)} t(i,k) = \prod_{k=1}^{n(i)} t(i,k) = r_i. \]
We obtain
\[ \left\langle y_1,\dots,y_{\ell}, x(i,1), i=1,\dots,n \, \middle\vert \, s_1,\dots,s_m, {\textstyle \prod}_{k=1}^{n(i)} t(i,k) \right\rangle \cong \frac{\pi \ast \langle x_1,\ldots,x_n \rangle}{(r_1,\ldots,r_n)}.\]
The latter group is exactly $\pi(\mathcal{P})$.
Finally, using Seifert--van Kampen, attaching the copies of $D^3 \times S^1$ from the upper boundaries of the 3-handles does not change the fundamental group, and so the inclusion induces an isomorphism
\[ \pi_1(N) \overset{\cong}\longleftarrow \pi_1\Big(M \# n(S^2 \times S^2) \sm \cup_{i=1}^n C_i \Big) \cong \pi(\mathcal{P}).\]
Tracing through the construction, the map induced by the inclusion
\[\pi_1(N) \cong \pi(\mathcal{P}) \longrightarrow \pi_1(W) \cong \pi\]
is identified under the isomorphisms with the projection given by sending each $x_i$ to $e$.

Let $x = [X] = [X(\mathcal{P})] \in \Wh(\pi)$. If $\dim(x) =0,2$, then it follows that the map $\pi_1(N) \to \pi_1(W)$ induced by the inclusion is an isomorphism, and so $W$ is an $h$-cobordism. By construction $\tau(W,M) = x$. On the other hand, if $\dim(x) =3$, then $\pi_1(N) \to \pi_1(W)$ is not an isomorphism, and so $W$ arising from the standard construction fails to be an $h$-cobordism. This completes the proof of \cref{theorem:obvious construction does not work}.

\begin{example}
For $\pi = C_5 = \langle g \mid g^5 \rangle$ and $x \in \Wh(\pi)$ represented by the unit $1-g+g^2$, we can assume $\mathcal{P}$ to have a single generator $x$ and a single relation $x(gx^{-1}g^{-1})(g^2xg^{-2})$. Then the above construction has outgoing boundary where $\pi_1(N)$ surjects onto $S_5$, so is not isomorphic to $C_5$ \cite[Example (4.5)]{MCohen}.  That is, $\pi(\mathcal{P})$ is a proper admissible extension. A computer search with GAP among similar admissible extensions of $C_5$ yielded no trivial examples.
\end{example}

\begin{remark}
We finish by commenting on how far a general $h$-cobordism is from the standard construction. Standard handle trading techniques yield the following: every $h$-cobordism $W \colon M \leadsto N$ can be assumed to only have $2$- and $3$-handles, with $2$-handles necessarily trivially attached, as many $3$-handles as $2$-handles, and the $3$-handles attached simultaneously. Thus it is determined by disjoint embeddings $\varphi_i \colon S^2 \hookrightarrow M \# n(S^2 \times S^2)$ with trivialised normal bundles. However, these need not be isotopic to the type of 2-spheres used in the standard construction. It is not apparent whether it is possible to extract from these 2-spheres an admissible presentation or relate its torsion to $\tau(W,M)$.
\end{remark}

\def\MR#1{}
\bibliography{biblio}
\end{document}